\documentclass{amsart}

\usepackage{latexsym}
\usepackage{amsfonts,amssymb} 
\usepackage{graphicx}

\def\cal{\mathcal}

\newcommand{\field}[1]{\mathbb{#1}}

\newcommand{\C}{\field{C}}

\newcommand{\N}{\field{N}}
\newcommand{\Q}{\field{Q}}

\newcommand{\Z}{\field{Z}}
\newcommand{\krn}{{\rm ker}\,}

\newtheorem{theorem}{Theorem}[section]
\newtheorem{proposition}{Proposition}[section]
\newtheorem{lemma}{Lemma}[section]
\newtheorem{corollary}{Corollary}[section]
\newtheorem{definition}{Definition}[section]

\newtheorem{example}{Example}[section]

\begin{document}

\makeatletter	   
\makeatother     

\title{An affine version of a theorem of Nagata}
\author{Gene Freudenburg}
\date{\today} 
\maketitle

\pagestyle{headings}

 \begin{abstract} Let $R$ be an affine $k$-domain over the field $k$. The paper's main result is that, if $R$ admits a non-trivial embedding in a polynomial ring 
 $K[s]$ for some field $K$ containing $k$, then $R$ can be embedded in a polynomial ring $F[t]$ which extends $R$ algebraically. This theorem can be applied to subrings of a ring which admits a non-zero locally nilpotent derivation.  In this way, we obtain a concise new proof of the cancellation theorem for rings of transcendence degree one for fields of characteristic zero.
\end{abstract} 
 
\section{Introduction} If $F\subset E$ are fields and $x\in E$, the subfield of $E$ generated by $F$ and $x$ is denoted $F(x)$. If $x$ is transcendental over $F$, then $F(x)$ is isomorphic to the field of rational functions in one variable over $F$, and we write $F(x)\cong F^{(1)}$.
In his 1967 paper \cite{Nagata.67}, Nagata proved the following fundamental result for fields.
\begin{theorem}{\rm (\cite{Nagata.67}, Thm. 2 and \cite{Ohm.88}, Thm. 5.2)}  Let $k, K,L$ be fields such that:
\begin{itemize}
\item [(i)] $k\subset K$ and $k\subset L\subset K^{(1)}$
\item [(ii)] $K$ is finitely generated over $k$
\item [(iii)] $L\not\subset K$
\end{itemize}
Then there exists a finite algebraic extension of the form $L\subset M^{(1)}$ for some field $M$ containing $k$.
\end{theorem}
This result extends the famous theorem of L\"uroth, which asserts that, if $k\subset L\subset k(x)$ are fields with $k\ne L$ and $x$ transcendental over $k$, then there exists $y\in k(x)$ with $L=k(y)$. By combining the theorems of L\"uroth and Nagata, we get an even stronger statement for fields of transcendence degree one over $k$; see the {\it Appendix}. 

We consider an analogous situation for integral domains. The polynomial ring in one variable $x$ over the field $F$ is denoted $F[x]=F^{[1]}$. For the integral domain $R$, we seek criteria to determine when $R=F^{[1]}$; when $R\subset F^{[1]}$ and $R\not\subset F$; or when $R\subset F^{[1]}$ with $F^{[1]}$ algebraic over $R$. 
Our main result is {\it Theorem~\ref{main}}, which may be regarded as an affine version of Nagata's theorem:
\begin{quote}{\it
Let $k$ be a field, and let $R$ be an affine $k$-algebra.
Suppose that there exists a field $K$ with
$R\subset K^{[1]}$ and $R\not\subset K$. Then there exists a field $F$ and an algebraic extension $R\subset F^{[1]}$. }
\end{quote}
This result is of particular interest in the setting of locally nilpotent derivations, where we assume that the ground field $k$ is of characteristic zero. If an integral $k$-domain $B$ admits a non-zero locally nilpotent derivation $D$, then $B$ is contained in $K[s]$, where $K$ is the field of fractions of the kernel of $D$ and $s$ is a local slice. 
Thus, any affine subalgebra $R\subset B$ not contained in the Makar-Limanov invariant of $B$ is isomorphic to a non-trivial subring of $F[s]$ for some field $F$, where $F[s]$ is algebraic over $R$. 

For rings of transcendence degree one over $k$,  {\it Thm.~\ref{main2}} gives an even stronger conclusion.
\begin{quote}{\it 
Let $k$ be a field, and let $R$ be a $k$-algebra with ${\rm tr.deg}_kR=1$. 
Suppose that there exists a field $K$ with
$R\subset K^{[1]}$ and $R\not\subset K$. Then $R$ is $k$-affine and there exists a field $F$ algebraic over $k$ with $R\subset F^{[1]}$. 
If $k$ is algebraically closed, then there exists $t\in {\rm frac}(R)$ with $R\subset k[t]$.}
\end{quote}
In \cite{Abhyankar.Eakin.Heinzer.72}, Abhyankar, Eakin, and Heinzer proved that if $R,S$ are integral domains of transcendence degree one over a field $k$ such that the polynomial rings $R[x_1,...,x_n]$ and $S[y_1,...,y_n]$ are isomorphic $k$-algebras, then $R$ and $S$ are isomorphic. 
In {\it Section \ref{cancel}}, we apply {\it Thm.~\ref{main2}}, together with the well-known theorems of Seidenberg and Vasconcelos on derivations, to obtain a short proof of this result in the case $k$ is of characteristic zero. Makar-Limanov gave a proof of this result 
for $k=\C$ in \cite{Makar-Limanov.98}, and we follow his idea to use the Makar-Limanov invariant. Other proofs are given in \cite{Crachiola.04, Crachiola.Makar-Limanov.05} for the case $k$ is algebraically closed. 

\subsection{Background} L\"uroth's Theorem was proved by L\"uroth for the field $k=\C$ in 1876, and for all fields by Steinitz in 1910 \cite{Luroth.1876, Steinitz.1910}. 
One generalization states that, if $k\subset L\subset k(x_1,...,x_n)$ and $L$ is of transcendence degree one over $k$, then $L=k(y)$. This was proved by Gordan for $k=\C$ in 1887, and for all fields by Igusa in 1951 \cite{Gordan.1887,Igusa.51}; other proofs appear in \cite{Nagata.67, Samuel.53}.
In 1894, Castelnuovo showed that, if $\C\subset L\subset\C (x_1,...,x_n)$ and $L$ is of transcendence degree two over $\C$, then $L=\C (y_1,y_2)$ \cite{Castelnuovo.1894}. Castelnuovo's result does not extend to non-algebraically closed ground fields, or to fields $L$ of higher transcendence degree. An excellent account of ruled fields and their variants can be found in \cite{Ohm.88}, including the theorem of Nagata (Thm. 5.2 of \cite{Ohm.88}). 

For polynomial rings, Evyatar and Zaks \cite{Evyatar.Zaks.70} showed that if $R$ is a PID and $k\subset R\subset k[x_1,...,x_n]$, then $R=k^{[1]}$; in \cite{Zaks.71}, Zaks generalized this to the case $R$ is a Dedekind domain.
And in \cite{Abhyankar.Eakin.Heinzer.72}, (2.5), Abhyankar, Eakin, and Heinzer showed that, if $k\subset R\subset k[x_1,...,x_n]$ and $R$ is of transcendence degree one over $k$, then $R$ is isomorphic to a subring of $k^{[1]}$. 

The Makar-Limanov invariant of a commutative ring (defined below) was introduced by Makar-Limanov in the mid-1990s, and he called it the ring of {\it absolute constants}. It is an important invariant in the study of affine rings, affine varieties, and their automorphisms. 

\subsection {Preliminaries} 
If $B$ is an integral domain, then ${\rm frac}(B)$ is the quotient field of $B$, and $B^{[n]}$ is the polynomial ring in $n$ variables over $B$. Given $f\in B$, $B_f$ denotes the localization $B[1/f]$. 
The set of derivations $D:B\to B$ is ${\rm Der}(B)$. 

If $A\subset B$ is a subring, the transcendence degree of $B$ over $A$, 
denoted ${\rm tr.deg}_AB$, will mean the transcendence degree of ${\rm frac}(B)$ over ${\rm frac}(A)$. The set of elements of $B$ algebraic over $A$ is denoted by ${\rm Alg}_AB$. 

Let $k$ be a field of characteristic zero, and $B$ an integral domain containing $k$. The set of $k$-derivations $D:B\to B$ is denoted ${\rm Der}_k(B)$. $D$ is said to be {\it locally nilpotent} if, to each $b\in B$, there exists $n\in\N$ with $D^nb=0$. The set of locally nilpotent derivations of $B$ is denoted ${\rm LND}(B)$. 
If $D\in {\rm LND}(B)$ is non-zero, and $A$ is the kernel of $D$, then $A$ is algebraically closed in $B$ and ${\rm tr.deg}_A(B)=1$. 

The {\it Makar-Limanov invariant} of $B$ is the intersection of all kernels of locally nilpotent derivations of $B$, denoted $ML(B)$. Note that $ML(B)$ is a $k$-algebra which is algebraically closed in $B$, and that any automorphism of $B$ maps $ML(B)$ into itself. 

An element $s\in B$ is a {\it local slice} of $D$ if $D^2s=0$ and $Ds\ne 0$. Note that every non-zero element of ${\rm LND}(B)$ admits a local slice. If $s\in B$ is a local slice of $D$, then: 
\begin{equation}\label{cylinder}
B_{Ds}=A_{Ds}[s]=(A_{Ds})^{[1]} \quad {\rm and }\quad D=d/ds
\end{equation}
This implies the following property: If $Df\in fB$ for some $f\in B$, then $Df=0$. 
The reader is referred to \cite{Freudenburg.06} for further details regarding locally nilpotent derivations. 

We also need the following.

\begin{proposition}\label{Nowicki} {\rm (\cite{Nowicki.94a}, Prop. 5.1.2)} Let $k$ be a field and $A$ a commutative $k$-algebra. Then, for any field extension $L/k$, 
$A$ is finitely generated over $k$ if and only if $L\otimes_kA$ is finitely generated over $L$. 
\end{proposition} 

\section{Main Theorem}

For a field $K$, the polynomial ring $K[s]=K^{[1]}$ is naturally $\Z$-graded over $K$, where $s$ is homogeneous of degree one. 
Let $\deg$ be the associated degree function in $s$ over $K$. A subring $R\subset K[s]$ is {\it homogeneous} if the $\Z$-grading restricts to $R$.

\begin{lemma}\label{homogeneous} Let $k\subset K$ be fields and $K[s]=K^{[1]}$. Suppose that $R\subset K[s]$ is a homogeneous subring with $k\subset R$ and $R\not\subset K$. Let $L={\rm frac}(R)\cap K$, and let $\hat{L},\hat{K}$ denote the algebraic closures of $L$ and $K$, respectively. Then there exists $c\in \hat{K}$ and integer $d\ge 1$ such that $R\subset \hat{L}[cs^d]$ and $\hat{L}[cs^d]$ is algebraic over $R$. 
\end{lemma} 

\begin{proof} Define the integer: 
\[
d=\gcd\{ \deg r\, |\, r\in R, \, r\ne 0\}
\]
Let homogeneous $r\in R$ of positive degree be given. Then there exist $\kappa\in K$ and positive $e\in\Z$ with $r=\kappa s^{de}$. Let $c\in\hat{K}$ be such that $c^{e}=\kappa$. Then 
$r=(cs^d)^{e}$.

If $\rho\in R$ is any other homogeneous element of positive degree, then $\rho =(c^{\prime}s^d)^{e^{\prime}}$ for $c^{\prime}\in\hat{K}$ and positive $e^{\prime}\in\Z$. We have:
\[
\frac{r^{e^{\prime}}}{\rho^e}=\frac{((cs^d)^e)^{e^{\prime}}}{((c^{\prime}s^d)^{e^{\prime}})^e} = \left(\frac{c}{c^{\prime}}\right)^{ee^{\prime}} \in L
\quad\Rightarrow\quad \frac{c}{c^{\prime}}\in\hat{L}
\quad\Rightarrow\quad \hat{L}[c^{\prime}s^d]=\hat{L}[cs^d]
\]
It follows that $R\subset \hat{L}[cs^d]$. 
\end{proof}

\begin{theorem}\label{main}  Let $k$ be a field, and let $R$ be an affine $k$-algebra.
Suppose that there exists a field $K$ with
$R\subset K^{[1]}$ and $R\not\subset K$. Then there exists a field $F$ and an algebraic extension $R\subset F^{[1]}$. 
\end{theorem} 

\begin{proof} Suppose that $R\subset K[s]=K^{[1]}$. For each $g\in K[s]$, let $\bar{g}$ denote the highest-degree homogeneous summand of $g$ as a polynomial in $s$. Define the set:
\[
 \bar{R}= \{ \bar{r}\, |\, r\in R\, ,\, r\ne 0\, \}
 \]
Then $k[\bar{R}]$ is a homogeneous subalgebra of $K[s]$ not contained in $K$.

By {\it Lemma~\ref{homogeneous}}, if $L={\rm frac}(k[\bar{R}])\cap K$
and if $\hat{L},\hat{K}$ are the algebraic closures of $L$ and $K$, respectively, then:
\begin{equation}\label{graded}
k[\bar{R}]\subset\hat{L}[cs^d] \quad (c\in \hat{K}\, ,\, d\ge 1)
\end{equation}

By  hypothesis, there exist $w_1,...,w_m\in R$ ($m\ge 1$) such that $R=k[w_1,...,w_m]$. 
Given $i$, assume that $w_i=\sum_{j=0}^{n_i}c_{ij}w^j$, where $c_{ij}\in K$. 
Define $A\subset \hat{K}$ and $B\subset \hat{K}[s]$ by:
\[
A=\hat{L}[c,c_{ij}\, |\, 1\le i\le m\, ,\, 0\le j\le n_i\, ] \quad {\rm and}\quad B=A[s]=A^{[1]}
\]
Then $R\subset B$, $A$ is finitely generated over $\hat{L}$, and the Jacobson radical of $A$ is trivial. 
Choose a maximal ideal $\mathfrak{m}$ of $A$ not containing $c$.

If $R\cap\mathfrak{m}B\ne (0)$, let non-zero $r\in R\cap\mathfrak{m}B$ be given. Since $\mathfrak{m}B=\mathfrak{m}[s]$, 
we have $r=\sum_{0\le i\le e}a_is^i$, where $a_i\in\mathfrak{m}$ for each $i$. Note that $e\ge 1$, since $\hat{L}\cap\mathfrak{m}=(0)$. 
Therefore, by equation (\ref{graded}), there exist $\epsilon\ge 1$ and non-zero $\lambda\in\hat{L}$ such that:
\[
\bar{r}=a_es^e = \lambda (cs^d)^{\epsilon}
\]
But then $c\in\mathfrak{m}$, a contradiction. Therefore, $R\cap\mathfrak{m}B = (0)$. 

Let $\pi :B\to B/\mathfrak{m}B$ be the canonical surjection of $\hat{L}$-algebras, noting that:
\[
B/\mathfrak{m}B = (A/\mathfrak{m}A)[\pi (s)]=\hat{L}^{[1]}
\]
Since $\pi (cs^d)=\pi (c)\pi (s)^d$, where $\pi (c)\ne 0$, we see that $\pi\vert_R$ is a degree-preserving isomorphism. It follows that
$R$ is a subring of $\hat{L}^{[1]}$ via $\pi$.

It remains to show that $R$ and $\hat{L}^{[1]}$ have the same transcendence degree over $k$. Since $R\subset \hat{L}^{[1]}$, it will suffice to show 
${\rm tr.deg}_k\hat{L}^{[1]}\le {\rm tr.deg}_kR$. By {\it Lemma~\ref{homogeneous}}, we see that ${\rm tr.deg}_k\hat{L}^{[1]}={\rm tr.deg}_kk[\bar{R}]$, so it will suffice to show 
${\rm tr.deg}_kk[\bar{R}]\le {\rm tr.deg}_kR$. 

Let $n=\dim_kR$, and let $r_1,...,r_{n+1}\in R$ be given. 
Then there exists a polynomial $h\in k[x_1,...,x_{n+1}]=k^{[n+1]}$ with $h(r_1,...,r_{n+1})=0$. If $k[x_1,...,x_{n+1}]$ is $\Z$-graded in such a way that each $x_i$ is homogeneous and the degree of $x_i$ is $\deg r_i$, then $H(\bar{r}_1,...,\bar{r}_{n+1})=0$, where $H$ is the highest-degree homogeneous summand of $h$. 
We have thus shown that any subset of $n+1$ elements in a generating set for $k[\bar{R}]$ is algebraically dependent over $k$. Therefore, the transcendence degree of $k[\bar{R}]$ over $k$ is at most $n$.

This completes the proof of the theorem.
\end{proof}

\section{Rings of Transcendence Degree One}

\begin{theorem}\label{main2}  Let $k$ be a field, and let $R$ be a $k$-algebra with ${\rm tr.deg}_kR=1$. 
Suppose that there exists a field $K$ with
$R\subset K^{[1]}$ and $R\not\subset K$. Then $R$ is $k$-affine and there exists a field $F$ algebraic over $k$ with $R\subset F^{[1]}$. 
If $k$ is algebraically closed, then there exists $t\in {\rm frac}(R)$ with $R\subset k[t]$.
\end{theorem} 

\begin{proof}
Suppose that $R\subset K[s]=K^{[1]}$, and let $\deg$ be the associated degree function in $s$ over $K$. 

Consider first the case $k$ is algebraically closed. The set 
\[
\Sigma :=\{\deg w\, |\, w\in R, w\ne 0\}\subset\N
\]
is a semigroup, and is therefore finitely generated as a semigroup.  Let $w_1,...,w_m\in R$ be such that
$\Sigma = <\deg w_1,...,\deg w_m>$, and define $S=k[w_1,...,w_m]\subset R$. Then, given $v\in R$, there exists $u\in S$ such that $\deg u=\deg v$. Assume that $\deg v\ge 1$. 

As in the preceding proof, since $u$ and $v$ are algebraically dependent over $k$, $\bar{u}$ and $\bar{v}$ are also algebraically dependent over $k$. Since $u$ and $v$ have the same degree,  there exists $P\in k[x,y]=k^{[2]}$ which is homogeneous relative to the standard $\Z$-grading of $k[x,y]$ such that $P(\bar{u},\bar{v})=0$. 
Write $P(x,y)=\Pi_{1\le i\le \ell}(\alpha_ix+\beta_iy)$, where $\ell$ is a positive integer and $\alpha_i,\beta_i\in k^*$ ($1\le i\le \ell$). Then 
$\alpha_i\bar{u}+\beta_i\bar{v}=0$ for some $i$. 
Therefore, $\deg (\alpha_iu+\beta_iv)<\deg v$ for some $i$. By induction on degrees, we can assume that  $\alpha_iu+\beta_iv\in S$, which implies $v\in S$, and $R=S$. Therefore, $R$ is finitely generated over $k$ when $k$ is algebraically closed. 

For general $k$, let $\hat{k},\hat{K}$ denote the algebraic closures of $k$ and $K$, respectively. Set $\hat{R}=\hat{k}\otimes_kR$. 
Then ${\rm tr.deg}_{\hat{k}}\hat{R}=1$, $\hat{R}\subset \hat{K}^{[1]}$ and $\hat{R}\not\subset \hat{K}$. By what was shown above, we conclude that $\hat{R}$ is affine over $\hat{k}$. Therefore, {\it Prop.~\ref{Nowicki}} implies that $R$ is affine over $k$. 

By {\it Thm.~\ref{main}}, there exists a field $F$ algebraic over $k$ with $R\subset F^{[1]}$. 
If $k$ is algebraically closed, then $F=k$ and $k\subset R\subset k[s]$ for some $s$ transcendental over $k$. If ${\cal O}$ is the integral closure of $R$ in ${\rm frac}(R)$, then since $k[s]$ is integrally closed, we have $k\subset R\subset {\cal O}\subset k[s]$. In this situation, it is known that ${\cal O}=k[\theta ]$ for some $\theta\in k[s]$; see \cite{Cohn.64}, Prop. 2.1.
\end{proof}

\begin{corollary}\label{Lenny}  {\rm (See \cite{Makar-Limanov.98}, Lemma 14)} Let $k$ be an algebraically closed field of characteristic zero, and let $B$ be a commutative $k$-domain. Given $r\in B$, if $r\not\in ML(B)$, then there exists $t\in {\rm frac}({\rm Alg}_{k[r]}B)$ such that ${\rm Alg}_{k[r]}B\subset k[t]$.
\end{corollary}

\begin{proof} By hypothesis, there exists $D\in {\rm LND}(B)$ with $Dr\ne 0$. If $A=\krn D$ and $K={\rm frac}(A)$, then $K\otimes_kB=K^{[1]}$ by equation (\ref{cylinder}).
We therefore have ${\rm Alg}_{k[r]}B\subset K^{[1]}$, and $r\not\in K$. The result now follows by {\it Thm.~\ref{main2}}.
\end{proof}

Makar-Limanov asked whether this result generalizes to rings of transcendence degree 2: Let $k$ be an algebraically closed field of characteristic zero, and let $B$ be a commutative $k$-domain. Given $x,y\in B$, does the following implication hold?
\[
{\rm Alg}_{k[x,y]}B\cap ML(B)=k \quad\Rightarrow\quad {\rm Alg}_{k[x,y]}B\subset k^{[2]}
\]
See \cite{Makar-Limanov.98}, p.39. 

\begin{example}\label{ex1} {\rm Let $k\subset K$ be fields, where $K=k[\alpha ]$ is a simple algebraic extension of $k$, and $[K:k]\ge 2$. Define:
\[
R=k[u,v]\subset K[s]=K^{[1]}
\]
where $u=\alpha s^2$ and $v=\alpha s^3$. Since $s=v/u$ and $\alpha =u^3/v^2$, we see that ${\rm frac}(R)=K(s)$. If $R\subset k[t]$ for $t\in {\rm frac}(R)$, then 
$k(t)={\rm frac}(R)=K(s)$, which is not possible. Therefore, the ring $R$ cannot be embedded in $k^{[1]}$. 
This shows that the hypothesis that the field $k$ be algebraically closed is necessary in the last statement of {\it Thm.~\ref{main2}}.}
\end{example} 

\begin{example} {\rm As an illustration of {\it Cor.~\ref{Lenny}}, let $k[x,y]=k^{[2]}$, and write $k[x,y]=\oplus_{i\ge 0}V_i$, where $V_i$ is the vector space of binary forms of degree $i$ over $k$. Define 
$D\in {\rm LND}(k[x,y])$ by $D=x\frac{\partial}{\partial y}$. Then $D$ is linear, meaning that $D(V_i)\subset V_i$ for each $i$. Therefore, if $B=k[V_2,V_3]$, then $D$ restricts to $B$. Let $R$ be the algebraic closure of $k[y^2]$ in $B$, noting that $D(y^2)\ne 0$. Then $R=k[y^2,y^3]$ and ${\rm frac}(R)=k(y)$. }
\end{example}


\section{Cancellation Theorem for Rings of Transcendence Degree One}\label{cancel}

\subsection{Integral Extensions and the Conductor Ideal}

\begin{definition} {\rm Let $A\subset B$ be integral domains. The {\it conductor of $B$ in $A$} is:
\[
{\cal C}_A(B)=\{ a\in A\, |\, aB\subset A\}
\]
If ${\cal O}$ is the integral closure of $A$ in  ${\rm frac}(A)$, then the {\it conductor ideal of $A$} is ${\cal C}_A({\cal O})$. 
}
\end{definition}
Note that ${\cal C}_A(B)$ is an ideal of both $A$ and $B$, and is the largest ideal of $B$ contained in $A$. The following two properties of the conductor are easily verified.
\begin{itemize}
\medskip
\item [(C.1)] \quad  ${\cal C}_{A^{[n]}}(B^{[n]})={\cal C}_A(B)\cdot B^{[n]}$ for every $n\ge 0$
\medskip
\item [(C.2)] \quad $D{\cal C}_A(B)\subset {\cal C}_A(B)$ for every $D\in {\rm Der}(B)$ restricting to $A$
\end{itemize}

\begin{lemma}\label{conductor} Let $k$ be a field, let $A$ be an integral domain containing $k$, and let $\mathfrak{C}\subset A$ be the conductor ideal of $A$. If $A$ is affine over $k$, then 
$\mathfrak{C}\ne (0)$.
\end{lemma}

\begin{proof} Since $A$ is affine over $k$, its normalization ${\cal O}$ is also affine over $k$, and is finitely generated as an $A$-module (see \cite{Hartshorne.77}, Ch. I, Thm. 3.9A). 
Let $\{\omega_1,...,\omega_n\}$ be a generating set for ${\cal O}$ as an $A$-module, and let non-zero $a\in A$ be such that $a\omega_1,...,a\omega_n \in A$. Then $a\in\mathfrak{C}$. 
\end{proof}

\begin{theorem} {\rm (Seidenberg \cite{Seidenberg.66})} Let $A$ be a noetherian integral domain containing $\Q$, 
and let ${\cal O}$ be the integral closure of $A$ in ${\rm frac}(A)$. Then every $D\in {\rm Der}(A)$ extends to ${\cal O}$. 
\end{theorem}

\begin{theorem} {\rm (Vasconcelos \cite{Vasconcelos.69})} Let $A\subset A^{\prime}$ be integral domains containing $\Q$, 
where $A^{\prime}$ is an integral extension of $A$. If $D\in {\rm LND}(A)$ extends to $D^{\prime}\in {\rm Der}(A^{\prime})$, then $D^{\prime}\in {\rm LND}(A^{\prime})$.
\end{theorem}

\subsection{The Theorem of Abhyankar, Eakin, and Heinzer}
\begin{theorem} {\rm (\cite{Abhyankar.Eakin.Heinzer.72}, (3.3))} Let $k$ be a field, and let $R,S$ be integral $k$-domains of transcendence degree one over $k$.
If $R^{[n]}\cong_kS^{[n]}$ for some $n\ge 0$, then $R\cong_kS$. 
\end{theorem}

\begin{proof} (characteristic $k=0$) Since $R$ is algebraically closed in $R^{[n]}$, we have:
\[
{\rm Alg}_k(R^{[n]})={\rm Alg}_k(R)
\]
Let $\alpha :R^{[n]}\to S^{[n]}$ be an isomorphism of $k$-algebras. If $k^{\prime}={\rm Alg}_k(R)$, then $\alpha (k^{\prime})={\rm Alg}_k(S)$, since $S$ is algebraically closed in $S^{[n]}$. Therefore,
identifying $k^{\prime}$ and $\alpha (k^{\prime})$, we can view $R$ and $S$ as $k^{\prime}$-algebras, and $\alpha$ as a $k^{\prime}$-isomorphism.  It thus suffices to assume that $k$ is algebraically closed in $R$. 

Since $ML(R^{[n]})\subset ML(R)$, we see that $ML(R^{[n]})$ is an algebraically closed subalgebra of $R$. Therefore, either $ML(R^{[n]})=R$ or $ML(R^{[n]})=k$.
\medskip

Case 1: $ML(R^{[n]})=R$. In this case, we also must have $ML(S^{[n]})=S$. Since $\alpha$ maps the Makar-Limanov invariant onto itself, we conclude that $\alpha (R)=S$.
\medskip

Case 2: $ML(R^{[n]})=k$. We will show that $R=k^{[1]}$ in this case. It suffices to assume that $k$ is an algebraically closed field: If $\hat{k}$ is the algebraic closure of $k$, and $\hat{R}=\hat{k}\otimes_kR$, then $ML(\hat{R}^{[n]})=\hat{k}$.  If this implies $\hat{R}=\hat{k}^{[1]}$, then $R=k^{[1]}$. (All forms of the affine line over a perfect field are trivial; see \cite{Russell.70}.) 

So assume that $k$ is algebraically closed. By hypothesis, there exists $D\in {\rm LND}(R^{[n]})$ with $DR\ne 0$. If ${\cal O}$ is the integral closure of $R$ in ${\rm frac}(R)$, then 
${\cal O}^{[n]}$ is the integral closure of $R^{[n]}$ in ${\rm frac}(R^{[n]})$. By property (C.1), if $\mathfrak{C}$ is the conductor ideal of $R$, then
$\mathfrak{C}\cdot {\cal O}^{[n]}$ is the conductor ideal of $R^{[n]}$. 

Let $s$ be a local slice of $D$, and let $K={\rm frac}(\krn D)$. 
Then by equation (\ref{cylinder}), $R\subset K[s]$ and $R\not\subset K$. 
By {\it Thm.~\ref{main2}}, $R$ is $k$-affine, and there exists $t\in {\rm frac}(R)$ such that ${\cal O}=k[t]$. 
By the theorems of Seidenberg and Vasconcelos, $D$ extends to a locally nilpotent derivation of ${\cal O}^{[n]}$; and by property (C.2), 
$D(\mathfrak{C}\cdot {\cal O}^{[n]})\subset\mathfrak{C}\cdot {\cal O}^{[n]}$.

By {\it Lemma~\ref{conductor}}, $\mathfrak{C}\ne 0$. Since $\mathfrak{C}$ is an ideal of $k[t]$, there exists non-zero $h\in R$ with $\mathfrak{C}=h\cdot k[t]$. Thus, $\mathfrak{C}\cdot O^{[n]}=h\cdot {\cal O}^{[n]}$ and $D(h\cdot {\cal O}^{[n]})\subset h\cdot {\cal O}^{[n]}$. Therefore, $Dh=0$.  
If $h\not\in k$, then $k[h]\subset\krn D$ implies $R\subset\krn D$, which is not the case. Therefore, $h\in k^*$ and $R=k[t]$. By symmetry, $S=k^{[1]}$.
\end{proof}


\section{Appendix}

Combining the theorems of L\"uroth and Nagata gives the following corollary.

\begin{corollary} Suppose that $k\subset L$ are fields, where $k$ is algebraically closed, $L$ is finitely generated over $k$, and ${\rm tr.deg}_kL=1$. If there exists a field $E$ containing $k$ such that $L\subset E^{(1)}$ and $L\not\subset E$, then $L=k^{(1)}$. 
\end{corollary}

\begin{proof} Assume that $L\subset E(s)=E^{(1)}$. Let $\alpha_1,...,\alpha_n\in L$ be such that $L=k(\alpha_1,...,\alpha_n)$.
Choose $f_i(s),g_i(s)\in E[s]$ such that $\alpha_i=f_i/g_i$, and let $K$ be the subfield of $E$ generated by the coefficients of $f_i$ and $g_i$, $1\le i\le n$. Then $K$ is finitely generated over $k$, and $L\subset K(s)$. By Nagata's theorem, there exists a finite algebraic extension $L\subset M^{(1)}$ for some field $M$ containing $k$. Since the transcendence degree of $L$ over $k$ is one, we see that $M$ is algebraic over $k$, i.e., $M=k$. The corollary now follows by L\"uroth's Theorem.
\end{proof}

We conclude by asking if the analogue of {\it Thm.~\ref{main}} holds for Laurent polynomial rings: Let $k$ be a field, and let $R$ be an affine $k$-algebra.
Suppose that there exists a field $K$ with
$R\subset K^{[\pm 1]}$ and $R\not\subset K$. Does it follow that there exists a field $F$ and an algebraic extension $R\subset F^{[\pm 1]}$?


\bigskip

\noindent \address{Department of Mathematics\\
Western Michigan University\\
Kalamazoo, Michigan 49008}\\
\email{gene.freudenburg@wmich.edu}
\bigskip

\end{document}